\documentclass[preprint,10pt]{elsarticle}

\usepackage{amssymb}
\usepackage{amsmath}
\usepackage{amsthm}
\usepackage[a4paper, left=2.5cm, right=2.5cm, top=2cm]{geometry}

\newtheorem{lemma}{Lemma}[section]
\newtheorem{thm}[lemma]{Theorem}

\newtheorem{hyp}[lemma]{Hypothesis}

\journal{Archiv der Mathematik}
\begin{document}

\setlength{\parindent}{0mm}

\newcommand{\m}{$\,\textrm{max}\,$}
\newcommand{\w}{\widehat}
\newcommand{\wi}{\widehat}
\newcommand{\ov}{\overline}
\newcommand{\N}{\mathbb{N}}
\def \P{\mathbb{P}}

\newcommand{\E}{\mathcal{E}}
\newcommand{\K}{\mathcal{K}}
\newcommand{\sym}{\textrm{Sym}}
\newcommand{\A}{\textrm{Alt}}

\newcommand{\Z}{\mathbb{Z}}

\newcommand{\wt}{\widetilde}
\newcommand{\wh}{\widehat}
\newcommand{\ti}{\tilde}

\newcommand{\M}{$\textrm{M}$}
\newcommand{\J}{$\textrm{J}$}
\newcommand{\ch}{$\textrm{char}$}
\newcommand{\sy}{$\,\textrm{Syl}$}
\newcommand{\au}{$\textrm{Aut}$}
\newcommand{\PSL}{$\textrm{PSL}$}
\newcommand{\PSU}{$\textrm{PSU}$}
\newcommand{\PGL}{$\textrm{PGL}$}
\newcommand{\PGaL}{P\Gamma L}
\newcommand{\GL}{$\textrm{GL}$}
\newcommand{\GU}{$\textrm{GU}$}
\newcommand{\Sp}{$\textrm{Sp}$}
\newcommand{\PSp}{$\textrm{PSp}$}
\newcommand{\Sz}{$\textrm{Sz}$}
\newcommand{\SL}{$\textrm{SL}$}
\newcommand{\SU}{$\textrm{SU}$}
\newcommand{\F}{$\textrm{GF}$}
\newcommand{\C}{$\textrm{C}$}
\newcommand{\FO}{\textrm{fix}_{\Omega}}
\newcommand{\FL}{\textrm{fix}_{\Lambda}}
\newcommand{\FD}{\textrm{fix}_{\Delta}}
\newcommand{\fixO}{\textrm{fix}_{\Omega}}
\newcommand{\fixL}{\textrm{fix}_{\Lambda}}
\newcommand{\out}{$\textrm{Out}$}
\newcommand{\Sym}{\textrm{Sym}}
\newcommand{\Alt}{\textrm{Alt}}
\newcommand{\rank}{$\textrm{r}$}
\newcommand{\He}{$\textrm{He}$}
\newcommand{\aut}{\textrm{Aut}}

\def \<{\langle }
\def \>{\rangle }
\def \L{\mathcal{L}}

\begin{frontmatter}

\title{Finite soluble groups that act with fixity 2 or 3}

\author{Paula H\"ahndel, Christoph M\"oller and Rebecca Waldecker}

\begin{abstract}
In this article
we prove results about
finite soluble groups that act with fixity 2 or 3. 
\end{abstract}

\end{frontmatter}


\section{Introduction}

\vspace{0.2cm}
While permutation groups that act with low fixity are interesting in their own right, from a purely group theoretic perspective, 
we became interested in them because of applications to Riemann surfaces.
As far as we know, the notion of \textbf{fixity} goes back to Ronse (see \cite{Ro1980}): We say that \textbf{a group $G$ has fixity $k \in \N$ on a set $\Omega$} if and only if
$k$ is the maximum number of fixed points of elements of $G^\#$ on $\Omega$.
More background on our motivation and on applications can be found in \cite{MW2}, which was the
first article that we published on groups that act with low fixity.
Since this project had started in 2012, Kay Magaard and the third author have received numerous questions about special cases of 
group actions with low fixity, and often about soluble groups.
Therefore, the main purpose of this paper is to prove results for soluble groups that act with fixity 2 or 3.

Here is the main hypothesis for the remainder of this paper:

\begin{hyp}\label{hypkfix}
Suppose that $G$ is a finite, transitive permutation
group with permutation domain $\Omega$. Suppose further that
$k \in \N$ and that $G$ acts with fixity~$k$ on $\Omega$.
\end{hyp}

\begin{thm}\label{solu2main}
Suppose that Hypothesis  \ref{hypkfix} holds, that $G$ is soluble and that $k=2$.
Then one of the following is true:

(1) $G$ has a regular normal subgroup.

(2) $G$ has a normal subgroup of index 2 with two orbits of equal length, acting as a Frobenius group on each of them.

(3) $G \cong \Alt_4$ or $G \cong \Sym_4$ and $|\Omega|=6$.

\end{thm}

\begin{thm}\label{fix3solu}
Suppose that Hypothesis  \ref{hypkfix} holds, that $G$ is soluble and that $k=3$.
Then one of the following is true:

(1) $G$ has a regular normal subgroup.

(2) $G$ has a normal subgroup $F$ of index 3 or 6 that has three orbits on $\Omega$, of equal length,
and such that $F$ acts on each of them as a Frobenius group.

(3) $|\Omega|=6$ and $G$ has structure $(C_3 \times C_3) : C_4$, with point
stabilisers isomorphic to $\Sym_3$. 
\end{thm}

We are able to keep the preliminaries short, referring to earlier work, and then we focus on results that will later be applied to soluble groups. Some of our analysis is phrased in a more general way, because it was not always necessary to suppose that the group is soluble, and this might have useful applications later. A natural case distinction comes from the cases of fixity 2 and fixity 3,
which we look at in seperate sections. We also want to mention that soluble groups of fixity 4 are also worth inverstigating, but we did not include our first results here because they are preliminary and, already, quite technical and intricate.

\begin{center}

\textbf{Acknowledgments}

\end{center}

We thank everyone who has asked us interesting questions about
groups that act with low fixity. This article has been written with our friend and colleague Kay Magaard
in mind.


\section{Preliminaries}

\vspace{0.2cm}

All groups in this article are meant to be finite, and we use standard notation for orbits and point stabilisers.
Let $\Omega$ be a finite set and suppose that a group $G$ acts on $\Omega$.
Then for all $\Delta
\subseteq \Omega$, all $g \in G$ and all $H \le G$ we let
$\FD(H):=\{\delta \in \Delta \mid \delta^h=\delta$ for all $h \in
H\}$ denote \textbf{the fixed point set of $H$ in $\Delta$} and abbreviate
$\FD(\<g\>)$ by $\FD(g)$.
For all $n \in \N$, we denote the cyclic group of order $n$ by
$\C_n$.\\

Next we collect a few facts from previous papers 
for reference in this article:

\begin{lemma}\label{normaliser}
Suppose that Hypothesis \ref{hypkfix} holds and let $\alpha \in \Omega$.

(i) If $1 \neq X \le G_{\alpha}$, then $|N_G(X):N_{G_\alpha}(X)| \le k$.
In particular, if
$\FO(X)=\{\alpha\}$, then
$N_G(X) \le G_\alpha$.

(ii) If $x \in G_{\alpha}^\#$, then $|C_G(x):C_{G_\alpha}(x)| \le k$. In particular, if
$\FO(x)=\{\alpha\}$, then
$C_G(x) \le G_\alpha$.


(iii) $|Z(G)|$ divides $k$.

\end{lemma}

\begin{proof}
We have proven these statements (and some more) in \cite{MW2corr} as Lemma 2.2.
\end{proof}


\section{Soluble groups that act with fixity 2}

\begin{hyp}\label{2solu}
Suppose that Hypothesis \ref{hypkfix} holds and that $k=2$.
Suppose further that $N$ is an abelian minimal normal subgroup of $G$ and let $p$ be a prime and $n \in \N$ be such that $|N|=p^n$.
We let $\bar \Omega$ denote the set of $N$-orbits on $\Omega$ and we set $\bar G:=G/N$.
\end{hyp}

\begin{lemma}\label{fix2sol1}
Suppose that Hypothesis \ref{2solu} holds.
Then one of the following holds:

(1) $N$ acts semi-regularly on $\Omega$ or

(2) $|N|=4$, the $N$-orbits have length 2 and $|\Omega| =6$. Moreover $G \cong \Alt_4$ or $\Sym_4$.
\end{lemma}

\begin{proof}
Suppose that (1) does not hold and let $\alpha \in \Omega$.
Then $N_\alpha \neq 1$ and $N_\alpha \neq N$ because $N \unlhd G$.
We note that $|\alpha^N|$ is a power of p and therefore the number of fixed points of $N_\alpha$ on $\alpha^N$
is divisible by $p$. Then our fixity 2 hypothesis forces $p=2$, and $N_\alpha$ fixes exactly two points.
If $|\alpha^N|>2$, then $|N| \ge 4$ and acts with fixity 2 on a set of size at least 4, which means that
$|Z(N)| \le 2$. But this is false because $N$ is abelian.
Now all $N$-orbits have size 2, which means that the point stabilisers have index 2 and intersect trivially, and therefore $|N|=4$
and $|\Omega| \le 6$. This implies (2).

If $|\Omega|=4$, then $G$ is isomorphic to a transitive subgroup of $\Sym_4$ and
$|\bar \Omega|=2$. 
We prove that $N$ is exactly the kernel of the action of $G$ on $\bar \Omega$:
If $g \in G^\#$ acts trivially on $\bar \Omega$, then it acts as a transposition or a double-transposition on $\Omega$, while stabilising the two $N$-orbits, both of which have size 2. There are only three possibilities for that, which gives that $g \in N$.
Consequently $\bar G$ acts faithfully on $\bar \Omega$.
Now $|G|=8$ and $G \cong D_8$, which does not have a minimal normal subgroup of order 4.

If $|\Omega|=6$, then $G$ is isomorphic to a transitive subgroup of $\Sym_6$ and
$|\bar \Omega|=3$. If $g \in G^\#$ acts trivially on $\bar \Omega$, then, because of the fixity 2 hypothesis, it acts as a double-transposition or as a triple-transposition on $\Omega$, while stabilising each $N$-orbit. There are only four possibilities for that, hence $g \in N$ and it follows that
$\bar G$ acts faithfully and transitively on $\bar \Omega$.
Now $\bar G \cong \Sym_3$ or $\bar G \cong \Alt_3$, and we recall that $N$ is elementary abelian of order 4.
Then  $G \cong \Sym_4$ or $G \cong \Alt_4$.
\end{proof}

\begin{lemma}\label{fix2sol2}
Suppose that Hypothesis \ref{2solu} holds and that
 $N$ acts semi-regularly on $\Omega$. Then $\bar G$ acts with fixity at most 2 on $\bar \Omega$.
 If $g \in G\setminus N$ is such that $\bar g \in \bar G$ fixes exactly two $N$-orbits, then it either fixes a unique point on each of them or
$N=Z(G)$ has order 2.
\end{lemma}

\begin{proof}
Assume for a contradiction that $g \in G\setminus N$ has prime order $r$ and is such that $\bar g$ fixes at least three distinct points $\bar \alpha$,
$\bar \beta$ and $\bar \gamma$
on $\bar \Omega$.
If $r \neq p$, then $g$ fixes at least one point on each $N$-orbit that it stabilises, which gives a contradiction.
Thus $r=p$ and, since $g \notin N$, we know that $H:=N\langle g \rangle$ has order $p^{n+1}$.
It follows that $|H_\alpha|=p$ and that the number of fixed points of $H_\alpha$ on $\bar \alpha$ is divisible by $p$. Then our fixity 2 hypothesis gives that $p=2$.
We conclude that each point stabiliser in $H$ has order 2 and fixes exactly two points, both in the same $N$-orbit.

The $2^n-1$ elements in $N^\#$ act fixed point freely on $\Omega$, which means that they act as a product of $\frac{|\Omega|}{2}$ transpositions, respectively.
Next we count the elements of $H$ that fix a point (and hence exactly two points) on $\bar \alpha$: These are the conjugates of $g$.
Since $H$ is a $2$-group and acts with fixity 2 on $\bar \alpha$, we know that $|Z(H)|=2$ and therefore
$C_H(g)=Z(H) \cdot \langle g \rangle$ has order 4, giving that $g$ has exactly $\frac{|H|}{4}=2^{n-1}$ conjugates in $H$.
In a similar way, we find $2^{n-1}$ elements that fix two points on $\bar \beta$ and on $\bar \gamma$, respectively.
In conclusion, we have $3 \cdot 2^{n-1}$ elements in $H$ that fix two points and $2^n-1$ elements that have no fixed points, giving
too many elements for $H$ in total.
This contradiction proves our first statement.

For the second statement we let $g \in G$ be such that $\bar g$ fixes exactly two $N$-orbits $\bar \alpha$ and $\bar \beta$. Again let $H:=N\langle g \rangle$.
First we suppose that there exists an element $h$ of prime order $r \neq p$ in $H$.
Since $|\bar \alpha|=|N|=p^n$ is not divisible by $r$, it follows that $h$ fixes at least one point. This is true for $\bar \beta$ as well, and then the fixity 2 hypothesis gives that $h$ fixes exactly one point on $\bar \alpha$ and on $\bar \beta$, respectively.
Of course $h \in H\setminus N$, which means that $h$ is $H$-conjugate into $\langle g \rangle$ and therefore $g$ fixes exactly one point on $\bar \alpha$ and $\bar \beta$, as stated.

Now suppose that $H$ is a $p$-group.
Since $H$ stabilises $\bar \alpha$, of size $|N|<|H|$, we have that $H_\alpha \neq 1$.
Let $h \in H_\alpha^\#$ be an element of order $p$. Then the number of fixed points of $h$ on $\bar \alpha$ is divisible by $p$, which forces $p=2$, and also, we see that $H_\alpha$ has exactly two fixed points, both in $\bar \alpha$, and regular orbits otherwise. This implies that $o(h)=2$ and in fact $|H_\alpha|=2$.
Since $|\bar \alpha|=N$, we conclude that $|H|=2 \cdot |N|$.
Let $\Delta:=\bar \alpha \cup\bar \beta$.
Then $H$ contains $\frac{|N|}{2}$ involutions that generate point stabilisers for the action on $\bar \alpha$ and the same number of (different!) involutions that generate point stabilisers for the action on $\bar \beta$.
Together with the involutions in $N$ this gives $2 \cdot 2^{n-1}+(2^n-1)$ involutions. This number is $2^{n+1}-1$ and exhausts the number of non-trivial elements of $H$.
Hence $H$ is elementary abelian. Lemma \ref{normaliser}, together with the fact that $N$ acts semi-regularly, gives that
$|H|=4$ and $|N|=2$. Then $N=Z(G)$, again as stated.

\end{proof}

\begin{lemma}\label{fix2sol3}
Suppose that Hypothesis \ref{2solu} holds,  that
 $N$ acts semi-regularly on~$\Omega$ and that $|\bar \Omega|=6$.
 Then $\bar G$ is not isomorphic to $\Sym_4$ or to $\Alt_4$.
\end{lemma}

\begin{proof}
	Assume for a contradiction that $\bar G$ is isomorphic to $\Sym_4$ or to $\Alt_4$.
We can see, for example using \texttt{GAP} (\cite{GAP}), that there are three types of transitive subgroups of $\Sym_6$ that
are isomorphic to $\Alt_4$ or $\Sym_4$, and each time there is an elementary abelian normal subgroup of order $4$ with three orbits of size 2 and such that each involution fixes exactly two points.
Therefore we choose $N \le V \unlhd G$ such that $\bar V \unlhd \bar G$ is elementary abelian of order 4.

First assume that $p \neq 2$. Then a Sylow $2$-subgroup $T$ of $V$ is isomorphic to~$\bar V$, and it acts coprimely on $N$, which gives that $N=\langle C_N(t) \mid t \in T^\# \rangle$.
Let $t \in T^\#$. Then $\bar t$ fixes a point on $\bar \Omega$, whence $t$ stabilises an $N$-orbit. 

Since this orbit has length $p^n$ and $p$ is odd, we see that $t$ fixes a point in $\Omega$. Then the fixity \(2\) action yields that a subgroup of $C_N(t)$ of index at most \(2\), hence \(C_N(t)\) itself, is contained in a point stabiliser (see Lemma \ref{normaliser}(ii)). This is impossible because $N$ acts semi-regularly.

We deduce that $p=2$ and therefore $N$ is a $2$-group.
If $v \in V\setminus N$ and if $\bar \alpha$ is a fixed point of $\bar v$ in $\bar \Omega$, then $|N\langle v \rangle|>|N|=|\bar \alpha|$ and therefore we may suppose that $v \in V_\alpha$.
Then $v$ fixes exactly two points on $\bar \alpha$ and $\langle v \rangle=V_\alpha$ has order~$2$.
Therefore $\alpha^V$ has size $2^{n+1}$. This gives $2^n$ pairs of points that are fixed by involutions in $V$, which gives $2^n$ involutions in $V$ for each $V$-orbit, and there are at least three of them, by the first paragraph above. Moreover, there are $2^n-1$ involutions in $N \le V$, without fixed points. This gives at least $3 \cdot 2^n+2^n-1$ involutions in $V$, a group of order $2^{n+2}$, and we conclude that
$V$ is elementary abelian.
Since $V$ acts faithfully and transitively on $\alpha^V$, and with fixity $2$,  this contradicts Lemma \ref{normaliser}(iii). 
\end{proof}

\vspace{0.5cm}
\underline{Proof of Theorem \ref{solu2main}:}

Assume that the theorem is false and choose $G$ to be a minimal counterexample.
By hypothesis $G$ is non-trivial and soluble, therefore it has a non-trivial minimal normal subgroup $N$ and we may suppose that Hypothesis \ref{2solu} holds, with all its notation.
If $N$ does not act semi-regularly, then we apply Lemma \ref{fix2sol1} and we see that (2) holds. This gives Case (3) of our theorem, which is a contradiction.
Therefore $N$ acts semi-regularly on $\Omega$.

If $|\bar \Omega|=1$, then $N$ acts regularly on $\Omega$, as in (1), which gives another contradiction.

If $|\bar \Omega|=2$, then we let $F \unlhd G$ denote the kernel of the action of $G$ on $\bar \Omega$. We note that $N \le F$ and that $|G:F|=2$, because $G$ is transitive and the $N$-orbits have equal lengths.
By hypothesis there is some $g \in G$ that fixes exactly two points $\alpha, \beta \in \Omega$.
Then $g$ stabilises the $N$-orbits $\bar \alpha$ and $\bar \beta$, whereas all elements in $G \setminus F$ interchange the $N$-orbits.
Hence $g \in F_\alpha$. If $\beta \notin \bar \alpha$, then $F$ acts as a Frobenius group on its orbits, as in (2).
This is a contradiction, and therefore $\beta \in \bar \alpha$.
Then Lemma \ref{fix2sol2} implies that $N=Z(G)$ has order 2, and it follows that $|\Omega|=4$.
Now $G$ is isomorphic to a transitive subgroup of $\Sym_4$, with centre of order 2, and this only leaves the possibility $G \cong D_8$. In particular $G$ has a regular normal subgroup, which gives (1).
Again, this case is impossible, and we continue with the case where $|\bar \Omega| \ge 3$.

Lemma \ref{fix2sol2} gives that $\bar G$ is transitive and acts with fixity at most 2 on $\bar \Omega$.
Let $\alpha \in \Omega$. Then $G_\alpha \neq 1$, and $G_\alpha$ stabilises $\bar \alpha$. Therefore $\bar G$ does not act regularly on $\bar \Omega$.
If $\bar G$ acts as a Frobenius group, with Frobenius kernel $\bar F$, and if $F \unlhd G$ is its full pre-image in $G$, then $F$ acts regularly on $\Omega$.
This is (1) again, which gives a contradiction.
Consequently $\bar G$ acts with fixity 2 on $\bar \Omega$.
We recall that $G$ is a minimal counterexample, which means that $\bar G$ satisfies one of the cases (1), (2) or (3).
Case (3) does not hold because of Lemma \ref{fix2sol3}.
If Case (1) holds, then the full pre-image of a regular normal subgroup of $\bar G$ is a regular normal subgroup of $G$, as in (1). This is not possible.
Therefore Case (2) holds, and we let $F \unlhd G$ be such that $\bar F$ acts as a Frobenius group on its two orbits
$\bar \Omega_1$ and $\bar \Omega_2$.
Let $\bar K$ be the Frobenius kernel of $\bar F$ in this action, and let $\bar H$ be a Frobenius complement.
Then $|\bar K|=|\bar \Omega_1|$ is coprime to $|\bar H|$. Let $a \in G$ be such that $\bar a \in \bar H$.
If $o(a)$ divides $|\bar \Omega_1|$, then
$o(\bar a)$, which is a divisor of $o(a)$, also does, and this contradicts the fact that
$|\bar H|$ and $|\bar K|$ have coprime orders, because they are a complement and the kernel of a Frobenius group. 
Therefore $o(a)$ does not divide $|\bar \Omega_1|$.
The element $a$ stabilises $\bar \Omega_1$, and it cannot act semi-regularly on $\bar \Omega_1$,
hence it must fix a point. The same holds for $\bar \Omega_2$, and then the fixity 2 hypothesis gives that $a$ fixes exactly two points in $\Omega$, one on each $F$-orbit.
Once more, this leads to Case (2), and this is our final contradiction.\\

\section{Soluble groups that act with fixity 3}

\begin{hyp}\label{3solu}
Suppose that Hypothesis \ref{hypkfix} holds and that $k=3$.
Suppose further that $N$ is a minimal normal subgroup of $G$ and let $p$ be a prime and $n \in \N$ be such that $|N|=p^n$.
We let $\bar \Omega$ denote the set of $N$-orbits on $\Omega$ and we set $\bar G:=G/N$.
\end{hyp}

\begin{lemma}\label{fix3sol1}
Suppose that Hypothesis \ref{3solu} holds.
Then $N$ acts semi-regularly on $\Omega$ or the following is true:

$|\Omega|=6$, $|N| = 9$, with $N$-orbits of length $3$, $G$ has structure $(C_3 \times C_3) : C_4$ and the point
stabilisers are isomorphic to $\Sym_3$.
\end{lemma}

\begin{proof}
We suppose that $\alpha \in \Omega$ is such that $N_\alpha \neq 1$.
If $N=N_\alpha$, then the transitive action of $G$ on $\Omega$ forces $N$ to fix all points in $\Omega$, which is impossible.
Therefore $\bar \alpha:=\alpha^N=|N:N_\alpha|$ has size at least $2$, and it is a power of $p$ because $N$ is a $p$-group.
Now the number of fixed points of $N_\alpha$ on $\bar \alpha$ is divisible by $p$,
which by our fixity hypothesis forces $p\in \{2,3\}$.

Case 1: $p=2$.

First we notice that, by Lemma \ref{normaliser}(ii), this implies that $|N| \ge 4$ and $N \cap Z(G)=1$.
Also, $N$ is an abelian $2$-group that acts with fixity $2$ on $\bar \alpha$.
Using Lemma \ref{normaliser}(ii) this is possible only if $|\bar \alpha|=2$. Since $G$ acts with fixity 3, there must be an element $\beta \in \Omega \setminus \bar \alpha$.
The transitive action of $G$ yields that $N_\beta$ is conjugate to $N_\alpha$, hence it fixes two points on $\bar \beta:=\beta^N$.
$N_\alpha$ and $N_\beta$ both have index 2 in $N$, and the fixity 3 hypothesis forces $N_\alpha \cap N_\beta=1$, which means that
$|N|=4$ and $|\Omega| \in \{4,6\}$.
First assume that $|N|=4$. Then $|\Omega|=6$ because otherwise any element of $G$ with three fixed points fixes everything.
Checking all possibilities for subgroups of $\Sym_6$ that act with fixity 3 (for example with \texttt{GAP} (\cite{GAP}) or using Lemma 2.7 in \cite{MW3}), we see that this does not occur, and hence this case is impossible.

Case 2: $p=3$.

Then similar arguments give that $|\bar \alpha|=3$, that $N_\alpha$ fixes three points on $\bar \alpha$ and that there exists an element
$\beta \in \Omega \setminus \bar \alpha$.
Now
$N_\alpha$ and $N_\beta$ both have index 3 in $N$, hence $N_\alpha \cap N_\beta=1$, which means that
$|N|=9$ and $|\Omega| \in \{6,9,12\}$.

If $|\Omega|=6$, then we can apply Lemma 2.7 in \cite{MW3}, where Case (3) must hold and gives exactly the second possibility in the lemma.

If $|\Omega|=9$, then $N$ has three orbits of size 3 and therefore it contains three subgroups of order 3 that each fix exactly one of these orbits point-wise. These three point stabilisers contain six non-trivial elements in total, which means that we can choose
$x \in N$ in such a way that it does not fix any point on $\Omega$. In particular $\langle x\rangle$ is a subgroup of $N$ of order 3, and with its fixed point profile it cannot be conjugate to any other subgroup of $N$ of order 3. This contradicts the fact that
$N$ is a minimal normal subgroup of $G$.

Finally, we assume that $|N|=9$ and $|\Omega|=12$. In particular $N$ has exactly four orbits on $\Omega$, each of size 3, and $G$ is isomorphic to a transitive subgroup of $\Sym_{12}$ with a minimal normal subgroup of order $9$.
We denote the $N$-orbits by $\bar \alpha, \bar \beta$, $\bar \gamma$ and $\bar \delta$ and we let $K$ denote the kernel of the action of $G$ on $\bar \Omega$. Then the fixity 3 hypothesis gives that each element in $K^\#$ must act non-trivially on at least three of the four elements of $\bar \Omega$, and at the same time $N \le K$.
Assume that $x\in K$ has prime order, but not order $3$. Then $x$ fixes at least one point on each $N$-orbit, which contradicts Hypothesis \ref{3solu}.
Thus $K$ is a normal $3$-subgroup of $G$.
This case can then be finished using \texttt{GAP} (\cite{GAP}). 
\end{proof}

The second case of this lemma does actually occur and is therefore mentioned as Case (3) of Theorem \ref{fix3solu}.
In the group $\Alt_6$, there is a subgroup with 
the given structure and action on $\{1,...,6\}$, namely $\langle (1,2,3), (2,5,3,6)(1,4)\rangle$ (with minimal normal subgroup $\langle (1,2,3),(4,5,6)\rangle$). 
From now on, we will therefore suppose that Hypothesis \ref{3solu} holds and that $N$ is semi-regular.

\begin{lemma}\label{fix3sol1a}
Suppose that Hypothesis \ref{3solu} holds and that $N$ acts semi-regularly.
Then $\bar G$ acts with fixity at most 3 on $\bar \Omega$.
\end{lemma}

\begin{proof}
Assume for a contradiction that $g \in G$ is such that $\bar g$ has prime order $r$ and fixes at least four points $\bar \alpha$, $\bar \beta$, $\bar \gamma$ and $\bar \delta$ on $\bar \Omega$.
Let $H \le G$ be the full pre-image of $\langle \bar g \rangle$. If $r \neq p$, then we let $h \in H$ be an element of order $r$ and
we note that $r$ does not divide the lengths of the orbit $\bar \alpha$, which means that $h$ must fix a point on it.
The same holds for the other fixed points of $\bar g$ on $\Omega$, which gives at least four fixed points of $h$ on $a$ and contradicts our fixity 3 hypothesis.

Therefore $r=p$ and we see that $H$ has order $p^{n+1}$.
We recall that $H$ stabilises the $N$-orbits $\bar \alpha$,...,$\bar \delta$ of size $p^n$, which implies that $H$ does not act regularly on these orbits. In fact, it follows that the point stabilisers in $H$ have order $p$ and that they each fix exactly p points, all in one $N$-orbit, respectively. Together with our fixity 3 hypothesis this forces $p \in \{2,3\}$, and there are $4 \cdot p^{n-1}$ possibilities for such point stabilisers in $H$.
Each of them has order $p$, hence contains $p-1$ non-trivial elements, and this gives us 
$(p-1) \cdot 4 \cdot p^{n-1}$ elements of order $p$ in $H$ that each fix $p$ points.
Moreover, we have $p^n-1$ elements of order $p$ that do not have any fixed points, coming from $N$.
In sum this gives too many elements of order $p$ in $H$, which means that we have a contradiction.
\end{proof}

\begin{lemma}\label{fix3sol1b}
Suppose that Hypothesis \ref{3solu} holds, that $N$ is semi-regular, and that
 $g \in G^\#$ and $\alpha \in \Omega$ are such that $g$ fixes exactly three points in $\bar \alpha$.
Then one of the following holds:

(1) $\bar \alpha$ is the unique $N$-orbit that is stabilised by $g$, or

(2) $|N| \in \{3,9\}$. 
\end{lemma}

\begin{proof}
Without loss $g$ has prime order $r$, by our fixity 3 hypothesis, and we suppose that (1) does not hold.
Then let $\beta \in \Omega$ be such that $\bar \beta$ is another $N$-orbit that is stabilised by $g$. Our fixity 3 hypothesis implies that $g$ acts without fixed points on $\bar \beta$.
Both orbits have length $|N|$ because $N$ acts semi-regularly, and now we see that
$|N|=|\bar \alpha|=p^n \equiv 3$ modulo $r$, but at the same time $r$ divides $|N|$.
Therefore $r=p=3$.
Let $P:=N \cdot \langle g \rangle$, which is a $3$-group, and assume for a contradiction that 
$|N| \ge 27$.
Then $|P| \ge 81$ and $P$ acts
transitively and with fixity 3 on $\bar \alpha$.
Then Lemma 2.20 in \cite{MW3} implies that $P$ has maximal class, with point stabilisers of order $3$ that fix exactly three points on $\bar \alpha$.

We can check with \texttt{GAP} (\cite{GAP}) that this is not possible if $|P|=81$.
Hence $|P| \ge 3^5$. We let $g \in M \le P$ be such that $|P:M|=3$, in particular $M \unlhd P$ and
$M$ is not transitive on $\bar \alpha$ because $g \in M$.
Now $\bar \alpha$ consists of three $M$-orbits, $N \cap M \unlhd M$ has index $3$ in $M$ and we can argue as above,
counting elements of order 3 that act fixed-point-freely or that fix exactly three points on $\bar \alpha$.
Then it follows that $M$ has exponent 3 and class at least 3, which is impossible.
\end{proof}

\begin{lemma}\label{fix3sol1c}
Suppose that Hypothesis \ref{3solu} holds, that $N$ is semi-regular and that
 $g \in G^\#$ fixes exactly three points $\alpha, \beta, \gamma$ in $\Omega$.
Then one of the following holds:

(1) $\beta,\gamma \in \bar \alpha$ or

(2) $\bar \alpha, \bar \beta$ and $\bar \gamma$ are pairwise distinct, which means that $\bar g$ fixes three points in $\bar \Omega$.
\end{lemma}

\begin{proof}
Suppose that (1) does not hold, without loss $\beta \notin \bar \alpha$, and suppose that $g$ has prime order $r$.
Then $g$ stabilises $\bar \beta$ as well, and it fixes one or two points on $\bar \alpha$, 
because of our global fixity 3 hypothesis. This means that
$|N|=|\bar \alpha| \equiv 1$ or $2$ modulo $r$.
If $\gamma \in \bar \alpha$ or $\gamma \in \bar \beta$, then $g$ has a unique fixed point on one of the orbits and two on the other, and hence $|N|$ is congruent to 1 and 2 modulo r at the same time. This is impossible.
Thus $\gamma \notin \bar \alpha \cup \bar \beta$, which gives (2).
\end{proof}

For the next lemma we recall that, if $k,l \in \N_0$ and a group acts as a $(k,l)$-group on a set, then this means that all non-trivial group elements have either exactly $k$ fixed points or exactly $l$ fixed points.

\begin{lemma}\label{fix3sol1d}
Suppose that Hypothesis \ref{3solu} holds, that $N$ is semi-regular, that $|\bar \Omega| \ge 2$ and that
$\bar G$ acts as a (0,2)-group on $\bar \Omega$.
Then $|N| \in \{3,9\}$.
\end{lemma}

\begin{proof}
By Hypothesis \ref{3solu} we have some element $g \in G^\#$ that fixes exactly three points on $\Omega$.
Then the corresponding $N$-orbits are stabilised by $g$, which means that $\bar g$ fixes exactly 
two points $\bar \alpha$ and $\bar \beta$ of $\bar \Omega$.
Then Lemma \ref{fix3sol1c} gives that the three fixed points of $g$ in $\Omega$ are all contained in one $N$-orbit 
(which is possibility (1) in the lemma).
Then Lemma \ref{fix3sol1b} is applicable, where (1) does not hold, and (2) is exactly our statement.
\end{proof}

\begin{lemma}\label{fix3solReg}
Suppose that Hypothesis \ref{3solu} holds. Then every one of the following conditions guarantees that $G$ has a regular normal subgroup:

(i) $|N|=3$ and $|\Omega|=9$.

(ii) $|\bar \Omega|\le 2$, including the cases $|N| \in \{3,9\}$ and $|\Omega| \in \{6,18\}$.

(iii) $N$ is semi-regular and $\bar G$ has a regular normal subgroup with respect to the action on $\bar \Omega$. 

\end{lemma}

\begin{proof}
In Case (i) we check the transitive subgroups of $\Sym_9$ that act with fixity 3, and then with \texttt{GAP} (\cite{GAP}) we can always find a regular normal subgroup.

In Case (ii) we begin with the case where $|\bar \Omega|=1$. Then $N$ is transitive, hence regular by Lemma \ref{fix3sol1} (because the second case there cannot occur now).
Next suppose that $|\bar \Omega|=2$.

Let $\alpha, \beta \in \Omega$ be such that $\bar \alpha$ and $\bar \beta$ are exactly the two elements in $\bar \Omega$.
Let $g \in G$ be of prime order and such that $g$ fixes exactly three points on $\Omega$, as guaranteed by our hypothesis.
Then $g$ cannot interchange $\bar \alpha$ and $\bar \beta$, but it stabilises them both, and then Case (2) of Lemma \ref{fix3sol1b} must hold. Now $|N| \in \{3,9\}$.
It follows that $\Omega$ has 6 or 18 elements, depending on whether $|N|=3$ or $|N|=9$.
Again we use \texttt{GAP} (\cite{GAP}) in order to check for transitive subgroups of $\Sym_6$ or $\Sym_{18}$ that act with fixity 3 and that have a minimal normal subgroup as in our situation, and then we obtain that $G$ has a regular normal subgroup.

In Case (iii) we suppose that $\bar R \unlhd \bar G$ acts regularly on $\bar \Omega$ and we let $R$ denote its full pre-image in $G$.
Then $R \unlhd G$. If $\alpha, \beta \in \Omega$ are in distinct $N$-orbits, then 
the transitive action of $\bar R$ gives an element $x \in R$ such that $\bar x$ maps $\bar \alpha$ to $\bar \beta$. Now $N\langle x \rangle \le R$ contains an element that maps $\alpha$ to $\beta$, which means that $R$ is transitive.
If $\beta \in \bar \alpha$, then we find an element in $N \le R$ that maps $\alpha$ to $\beta$. 

Finally, if $\alpha \in \Omega$ and $y \in R_\alpha$, then $\bar y$ fixes the point $\bar \alpha$ or $\bar \Omega$, which implies that $\bar y =1$ and $y \in N$. Since $N_\alpha=1$ by hypothesis in (iii), we deduce that all point stabilisers in $R$ are trivial and that $R$ is indeed a regular normal subgroup of $G$. 

\end{proof}

\begin{lemma}\label{fix3solSC1}
Suppose that Hypothesis \ref{3solu} holds, that $N$ is semi-regular and that $|\bar \Omega|=3$.
Then $G$ has a normal subgroup that is regular or that has index 3 or 6, with three orbits and Frobenius group action.
\end{lemma}

\begin{proof}
We let $K \unlhd G$ denote the kernel of the action of $G$ on $\bar \Omega$.
We know that $N \le K$ and that $G/K$ is isomorphic to a transitive subgroup of $\Sym_3$, because $G$ is transitive on $\Omega$.
Let $R \le G$ be such that $K \le R$ and $|R/K|=3$.
If $x\in R$ fixes a point $\alpha \in \Omega$, then it stabilises the orbit $\bar \alpha$, and therefore it cannot interchange the three elements of $\bar \Omega$ in a 3-cycle. Therefore $x\in K$ stabilises all three orbits, and we deduce that the elements in $R \setminus K$ do not fix any points of $\Omega$. If $K$ acts regularly on each of its orbits, then $R$ acts regularly on all of $\Omega$, and it has index 1 or 2 in $G$, which is one of our possibilities.

Suppose that $K$ does not act regularly on its orbits and let $y\in K$ be an element of prime order $r$ such that $y$ fixes a point $\alpha \in \Omega$. Then $y$ stabilises $\bar \alpha$ and counting gives two possibilities:
$y$ has exactly one fixed point on each $N$-orbit, and then $K$ acts as a Frobenius group on each $N$-orbit, leading to one of our possibilities again.
Or $y$ has order 2 or 3, with two or three fixed points on $\bar \alpha$ and no fixed points on the remaining $N$-orbits.

If $o(y)=2$, then $N$ and $K$ are $2$-groups, and they act with fixity 2 on $\bar \alpha$.
Since $|Z(G)|$ divides $3$ by Lemma \ref{normaliser}(iii), we know that $|N| \neq 2$ and therefore $|\bar \alpha| \ge 4$. Now $|N\langle y \rangle|=2^{n+1}$ and the $2^n-1$ involutions from $N$ act without fixed points, whereas there must be $3 \cdot 2^{n-1}$ pairs of points that are fixed by the conjugates of $y$. Together, this gives $2^{n+1}-1$ involutions, which implies that $N\langle y \rangle$ is elementary abelian, contrary to the fact that this group has a centre of order at most 2 (Lemma \ref{normaliser}(ii)).

If $o(y)=3$, then $N$ and $K$ are $3$-groups and we argue in a similar way, in parallel to Lemma \ref{fix3sol1b}.
If $|N| \le 9$, then we can use \texttt{GAP} (\cite{GAP}) and check the possibilities, and we obtain that the first case of our lemma holds.
If $|N| =27$, then we count the elements of order 3 in $N \cdot \langle y \rangle$ depending on whether they have three fixed points ($\frac{3 \cdot |N|}{3}$ possibilities for conjugates of $y$) or none ($|N|-1$ elements).
Together this gives that all elements in $N \cdot \langle y \rangle$ have order at most 3, and then we can use \texttt{GAP} (\cite{GAP}) again to check that this is impossible for a $3$-group that acts with fixity 3.
Now $|N| \ge 81$ and we can argue as in the corresponding case of Lemma \ref{fix3sol1b}, leading to a contradiction.
\end{proof}

\begin{lemma}\label{fix3solSC2}
Suppose that Hypothesis \ref{3solu} holds, that $|\bar \Omega|\ge 3$ and that $\bar G$ acts with fixity 2 on $\bar \Omega$.
Then $G$ has a regular normal subgroup.
\end{lemma}

\begin{proof}
We apply Theorem \ref{solu2main} to $\bar G$ and $\bar \Omega$ and look at the three cases. Case (i) immediately gives our claim, because of Lemma \ref{fix3solReg}. 

Next we consider Case (ii), and we let $\bar F$ denote a normal subgroup of $\bar G$ 
of index 2, with two orbits on $\bar \Omega$, acting as a Frobenius group on each of them. 
This means that $\bar G$ acts as a (0,2)-group on $\bar \Omega$, and Lemma \ref{fix3sol1d}
gives that $|N| \in \{3,9\}$.
Let $F$ denote the full pre-image of $\bar F$ in $G$ and let $K$ denote the full pre-image of the Frobenius kernel of $\bar F$. 
Finally, let $\bar H$ be a Frobenius complement in $\bar F$, with full pre-image $H$ in $G$, 
and let $M \le G$ be the full pre-image of $N_{\bar G}(\bar H)$ in $G$.
Then $|M:H|=2$ because $\bar G$ acts by conjugation on the set of Frobenius complements of $\bar F$, and we deduce that $\bar G=\bar K \rtimes \bar M$.
Let $\bar \alpha$ and $\bar \beta \in \bar \Omega$ be the two fixed points of $\bar H$, one in each $\bar F$-orbit.
Then $\bar M$ interchanges them, which implies that $M$ is transitive on the set $\Delta:=\bar \alpha \cup \bar \beta$. This set has size $2 \cdot |N| \in \{6,18\}$, and $M$ acts with fixity at most 3 on it. 
Now we can use \texttt{GAP} (\cite{GAP}) and find a regular subgroup $R$ of $M$ in the action on $\Delta$. 
Then $\bar K \rtimes \bar R$ is a regular normal subgroup of $\bar G$ and Lemma \ref{fix3solReg} gives our statement.

Finally, we prove that Case (iii) does not even occur. Otherwise $|\bar \Omega|=6$, $\bar G$ is isomorphic to $\Alt_4$ or $\Sym_4$, and we recall that $G$ acts with fixity 3 on $\Omega$. Let $g \in G$ be an element of prime order $r$ and with exactly three fixed points.
Then $g \notin N$, but $g$ stabilises at least one $N$-orbit and therefore $1 \neq \bar g$ is contained in a point stabiliser in $\bar G$.
The way that $\bar G$ acts on a set of size 6 with fixity 2 implies that the point stabilisers are $2$-groups and that $\bar G$ acts as a $(0,2)$-group. 
In particular $\bar g$ is a 2-element, it fixes exactly two points $\bar \alpha$ and $\bar \beta$ on $\bar \Omega$ and $o(g)=2$.
The three fixed points of $g$ on $\Omega$ are all contained in one $N$-orbit, say $\bar \alpha$,
which forces $|\bar \alpha|=|N|$ to be congruent to 3 modulo $2$, but also divisible by $2$.
This is impossible.
\end{proof}

\vspace{0.5cm}
\underline{Proof of Theorem \ref{fix3solu}:}

Let $N$ be a minimal normal subgroup of $G$. Since $G$ is non-trivial and soluble, we know that $N$ is non-trivial and elementary abelian, and we may suppose that Hypothesis \ref{3solu} holds, including all the notation there.
If the second case of Lemma \ref{fix3sol1} is true, then we can stop because this is included as (3) in the theorem.
Otherwise $N$ acts semi-regularly and $\bar G$ acts with fixity at most 3 on $\bar \Omega$ by Lemma \ref{fix3sol1a}.

If $|\bar \Omega|\le 2$, then Lemma \ref{fix3solReg}(ii) implies our claim.
If $|\bar \Omega| = 3$, then we apply Lemma \ref{fix3solSC1}.

From now on we suppose that $|\bar \Omega| \ge 4$, and we recall 
 that $\bar G$ acts with fixity at most 3 on $\bar \Omega$.
It does not act regularly because $G$ has non-trivial point stabilisers.
If $\bar G$ acts as a Frobenius group with Frobenius kernel $\bar K$, 
then $\bar K$ is a regular normal subgroup of $\bar G$. Then Lemma \ref{fix3solReg}(iii) yields that (1) holds.
If $\bar G$ acts with fixity 2, then Lemma \ref{fix3solSC2} gives that $G$ has a regular normal subgroup.

Finally, we are left with the case where $\bar G$ acts with fixity 3 on $\bar \Omega$.
Still, $|\bar \Omega| \ge 4$, and now we assume for a contradiction that the theorem is false.
We let $G$ be a minimal counterexample and we apply the theorem to $\bar G$.
If Case (1) holds, then $\bar G$ has a regular normal subgroup,and  then so does $G$, by Lemma \ref{fix3solReg} (iii).
Case (2) of our theorem leads to a more detailed analysis for $\bar G$ and $G$, which is why we leave this case for last.

If Case (3) of the theorem holds, then $|\bar\Omega|=6$ and $\bar G$ has structure $(C_3 \times C_3) : C_4$.
Since we know the points stabiliser structure as well, we have details about the fixed points of elements in $\bar G$:
Fixity 3 is exhibited by elements of order $3$ of $\bar G$ that stabilise exactly three $N$-orbits and interchange the reminaing three, and elements of order $2$ in the point stabilisers have exactly two fixed points.
Next we investigate the consequences for $N$ in this situation. Given that $\bar G$ is a $\{2,3\}$-group, it is natural to distinguish the cases where $N$ is a $2$-group, a $3$-group or a $p$-group for some prime $p\ge 5$.

First assume that $p=2$.
If $|N|=2$, then $|\Omega|\ge 8$ and we recall that some element of order $3$ of $\bar G$ fixes three $N$-orbits, and then a pre-image has to fix all six elements of $\Omega$ in these orbits. This contradicts our main fixity hypothesis.   
If $|N| \ge 4$, then we let $n \in \N$ be such that $|N|=2^n$ and we let
$\alpha \in \Omega$ be such that $s\in G_\alpha$ is a $2$-element and $\bar s$ fixes two $N$-orbits. Let $\beta \in \Omega$ be such that $s$ stabilises $\bar \alpha$ and $\bar \beta$. 
We note that $s$ is a $2$-element and that $\bar \alpha$ has $|N|$ elements, which forces $s$ to have exactly two fixed points on $\alpha^N$ and no other fixed points on $\Omega$. In particular $s$ acts like a transposition on $\Omega$. 
At the same time, it is the square of an element of order 4 (from the structure of $\bar G$ and because $N$ acts semi-regularly), and this is a contradiction. 

Next assume that $p=3$.
We let $\alpha \in \Omega$ be such that $x\in G_\alpha$ is a $3$-element and $\bar x$ fixes three $N$-orbits.
Since the fixed points of $x$ in $\Omega$ must lie in the fixed $N$-orbits, this situation forces $x$ to fix three points in one of the $N$-orbits and otherwise to have regular orbits. Now Lemma \ref{fix3sol1b}~(b) must hold, i.e. $|N| \in \{3, 9\}$. 

If $|N| = 3$,
then we first note that $|G:C_G(N)|\le 2$. We recall the structure of $\bar G$: The elements of order 4 square to elements that fix exactly two points in $\bar \Omega$. Hence if $g \in G$ has order 4, then $g^2$ has two fixed points, but also $g^2 \in C_G(N)$.
Since $|N|=3$, $g^2$ fixes a point in the stabilised $N$-orbits, and then $N$ fixes these points as well. This is a contradiction because $N$ acts semi-regularly. 
Now $|N| = 9$ and we let $M \unlhd G$ be such that $|G:M|=4$, again using the structure of $\bar G$.
In particular $M \in \sy_3(G)$, $|M|=81$ and we recall that fixity 3 is exhibited by $3$-elements. 
Lemma 2.20 of \cite{MW3} gives that $M$ has maximal class, and we see that $Z(M)$ has order $3$ and is contained in $N$.
On the other hand $Z(M) \unlhd G$, which means that $N$ is not a minimal subgroup. This is a contradiction. 

We conclude that $p \ge 5$. Now let $\alpha \in \Omega$ and let $x \in G_\alpha^\#$.
Then $x$ fixes two or three points, and therefore $C_N(x)=1$.
On the other hand, $G_\alpha \cong \bar G_{\bar \alpha} \cong \Sym_3$ is a Frobenius group that acts by conjugation on the abelian group $N$, with coprime orders. This means that 8.3.5 from \cite{KS} is applicable, and it gives a contradiction.\\

Finally, we turn to 
Case (2) of our theorem for $\bar G$. Then we let $\bar F \unlhd \bar G$ be as described there, and we let $F \unlhd G$ denote the full pre-image of $F$.
Then $|G:F| \in \{3,6\}$ and therefore, if we let $R \le G$ denote the full pre-image of a Sylow 3-subgroup of $G/F$, then it has index 1 or 2 in $G$. In particular $R \unlhd G$.

Suppose that $F$ acts semi-regularly on $\Omega$.
Since $N \le F$, the $F$-orbits on $\Omega$ are unions of $N$-orbits, and then the transitivity of $G$ on 
$\Omega$ implies that there are three $F$-orbits and that the elements in    
$R \setminus F$ interchange them in a $3$-cycle. 
Together this gives that $R$ acts regularly on $\Omega$, and (1) holds.

Now suppose that there is some $\alpha \in \Omega$ such that $F_\alpha \neq 1$.
Then $x \in F_\alpha$ stabilises $\bar \alpha$ and the $F$-orbit $\bar \alpha^{\bar F}$.
In particular $\bar x$ is contained in one of the Frobenius complements in $\bar F$, which means that it fixes exactly one point in each $\bar F$-orbit, and therefore $x$ stabilises three $N$-orbits in total.
This gives two possibilities:
$x$ fixes exactly one point on each $N$-orbit that it stabilises, or it has order 3 and fixes three points in $\bar \alpha$. In particular $F$ acts like a $(0,3)$-group on $\Omega$.
In the first case $F$ acts like a Frobenius group on each of its orbits, as in (2).
In the second case we turn to Lemma \ref{fix3sol1b}.
We recall that $x$ stabilises three $N$-orbits, with three fixed points on one of them, and then the lemma implies that $|N| \in \{3,9\}$.

Case 1: $|N|=3$.

Then we argue similarly to Lemma \ref{fix3solSC2}.
Hence we let $\bar K$ denote the Frobenius kernel of $\bar F$, with full pre-image $K$ in $G$, 
we let $\bar H$ be a Frobenius complement in $\bar F$, with full pre-image $H$ in $G$, and we let $M \le G$ be the full pre-image of $N_{\bar G}(\bar H)$ in $G$.
Then $\bar G=\bar K \rtimes \bar M$ and $\bar H$ fixes exactly three $N$-orbits.
Their union $\Delta$ has size 9, $M$ acts faithfully, transitively and with fixity 3 on it, and then 
we can see with \texttt{GAP} (\cite{GAP}) that $M$ has a normal subgroup $R$ that acts regularly on $\Delta$. 
Then $\bar K \rtimes \bar R$ is a regular normal subgroup of $\bar G$ and Lemma \ref{fix3solReg} gives our statement.

Case 2: $|N|=9$.

We recall that  $N$ acts semi-regularly.
Let $C:=C_G(N)$ and assume that some $c\in C^\#$ fixes a point $\omega \in \Omega$. 
Then a subgroup of index at most $3$ of $N$ is contained in $G_\omega$ (by Lemma \ref{normaliser}(i)), which is impossible.
Thus $C$ acts semi-regularly on $\Omega$, and $G/C$ is isomorphic to a subgroup of $\aut(N) \cong \GL_2(3)$. We started the present case with $x \in F_\alpha^\#$ and we saw that $o(x)=3$ and that $x$ fixes exactly three points in an $N$-orbit.
This means that, if we consider the action of $G/C$ on the set of $C$-orbits, then the point stabilisers have order divisible by $3$, hence they contain a Sylow $3$-subgroup of $G/C$, respectively.
The index of the point stabilisers in $|G/C|$ divides $2^4$ because $|\GL_2(3)|=2^4 \cdot 3$, which means that the size of the set that $G/C$ acts on divides $2^4$. We conclude that the number of $C$-orbits on $\Omega$  is $2^k$, where $1 \le k \le 4$.
This implies that there is no $F$-orbit that is contained in a $C$-orbit.

Conversely, assume for a contradiction that $\alpha \in \Omega$ is such that 
$\alpha^C \subseteq \alpha^F$. Then $C$ stabilises the set $\alpha^F$ and acts semi-regularly on it. 
If $d \in \N$ is the number of $C$-orbits in $\alpha^F$, then it follows that $\Omega$ consists of $3\cdot d=2^k$ $C$-orbits, which is not possible.
In particular $C \nleq F$, and therefore $C \cap F \neq C$.
Since $|G:F| \in \{3,6\}$, it follows that $|C:C \cap F| \in \{3,6\}$ as well.
We recall that $F$ and $C$ are normal subgroups of $G$, 
both containing $N$.
This implies that $C$ permutes the set of $F$-orbits, which are unions of $N$-orbits, and vice versa.
Moreover, the previous paragraph gives that $C/C \cap F$ permutes the three $F$-orbits regularly, and then $|C:C \cap F|=3$.
 
We recall that $\bar F$ acts as a Frobenius group, and we let $\bar K$ denote its Frobenius kernel and $K$ the full pre-image in $G$. Then the $F$-orbits are exactly the $K$-orbits, 
$CK \unlhd G$ and we will prove that $CK$ acts regularly, as in (1).
For this it suffices to prove that 
$|CK|=|\Omega|$.
We note that $C \cap F=C \cap K$, because $C$ acts semi-regularly on $\Omega$, and then 
$|CK|=\frac{|C| \cdot |K|}{|C \cap K|}=3 \cdot |K|=|\Omega|$.


This concludes the proof.




\begin{thebibliography}{0.2cm}


\bibitem{GAP}
  The GAP~Group, \textit{GAP -- Groups, Algorithms, and Programming,
  Version 4.12.2}.




\bibitem{MW2corr}
H\"ahndel, Paula and Waldecker, Rebecca:
Corrigendum and addendum to ``Transitive
permutation groups where nontrivial elements have at
most two fixed points'', arxiv version.




\bibitem{MW2}
	Magaard, Kay and Waldecker, Rebecca: Transitive permutation
	groups where nontrivial elements have at most two fixed points, 
	\emph{J. Pure Appl. Algebra} 219.4, 729--759. 


	
\bibitem{MW3}
	Magaard, Kay and Waldecker, Rebecca: 
	Transitive permutation groups with trivial four point stabilizers, \textit{J. Group Theory} 18 no. 5 (2015), 687--740.
	


\bibitem{KS}
Kurzweil, Hans und Stellmacher, Bernd: \textit{Theory of Finite Groups.} Springer, 2004.


\bibitem{Ro1980}
Ronse, Chr.: On Permutation Groups of Prime Power Order,
\textit{Math. Z.} \textbf{173}, 211--215 (1980).








\end{thebibliography}
\end{document}